\documentclass[10pt,twoside,english,reqno,a4paper]{amsart}

\usepackage[footskip=10mm]{geometry}
\usepackage{datetime}
\usepackage{fancyhdr}
\usepackage{amsmath}
\newlength{\defbaselineskip}
\usepackage{listings,graphicx,varioref,color,bm,stmaryrd}

\usepackage{amsmath}

\usepackage{amsfonts}
\usepackage{calrsfs}

\usepackage{amssymb}
\usepackage{mathrsfs}
\usepackage{makecell}
\usepackage[mathscr]{euscript}

\usepackage{mathpazo} 
\usepackage[scaled=.95]{helvet}
\usepackage{courier}

\usepackage{amscd}

\usepackage[T1]{fontenc}
\usepackage{graphics}
\usepackage[english]{babel}
\usepackage[utf8]{inputenc}
\usepackage[makeroom]{cancel}
\usepackage{pifont}
\usepackage{marvosym}
\usepackage{array}
\usepackage{datetime}
\definecolor{caribbeangreen}{rgb}{0.0, 0.8, 0.6}
\definecolor{darkpastelgreen}{rgb}{0.01, 0.75, 0.24}
\definecolor{green(pigment)}{rgb}{0.0, 0.65, 0.31}
\usepackage[dvipsnames]{xcolor}
\usepackage{hyperref}
\hypersetup{
colorlinks=true,
linkcolor= red!70!black,
citecolor= Orange,
urlcolor=Blue, 
}

\addto\extrasenglish{%
}

\addto\extrasenglish{%
}

\addto\extrasenglish{%
}

\usepackage[capitalize,nameinlink,noabbrev,nosort]{cleveref} 

\crefname{section}{ Section}{ Sections}
\crefname{subsection}{ Subsection}{ Subsections}
\crefname{appendix}{ Appendix}{ Appendices}

\crefname{figure}{ Figure}{ Figures}

\crefname{equation}{}{}

\crefname{definition}{Definition}{Definitions}
\crefname{theorem}{ Theorem}{ Theorems}
\crefname{proposition}{ Proposition}{ Propositions}
\crefname{corollary}{ Corollary}{ Corollaries}
\crefname{remark}{ Remark}{ Remarks}
\crefname{lemma}{ Lemma}{ Lemmata}
\crefname{theoa}{ Theorem}{ Theorems}
\crefname{lemb}{ Lemma}{ Lemmata}

\usepackage{tikz-cd}
\usetikzlibrary{patterns}
\usepackage{enumerate}
\usepackage{graphics}
\usepackage{pgfplots}
\pgfplotsset{width=7cm,compat=newest}
\usepackage{graphics}
\usepackage{tikz}
\usetikzlibrary{shapes.geometric}
\usepackage{tikz-cd}
\usepackage{caption}
\usepackage{subcaption}

\usepackage{float}

\newtheoremstyle{mytheoremstyle} 
{.5em}                    
{0cm}                    
{\slshape}                   
{}                           
{\bf}             
{.}                          
{.5em}                       
{}  

\theoremstyle{mytheoremstyle}
\newtheorem{theorem}{Theorem}[section]
\newtheorem{proposition}[theorem]{Proposition}
\newtheorem{lemma}[theorem]{Lemma}
\newtheorem{corollary}[theorem]{Corollary}
\newtheorem{definition}[theorem]{Definition}
\newtheorem{remark}[theorem]{Remark}

\numberwithin{equation}{section}
\long\def\salta#1{\relax}

\usepackage{kantlipsum}
\allowdisplaybreaks

\usepackage[normalem]{ulem}               
\newcommand\redout{\bgroup\markoverwith
{\textcolor{bor}{\rule[0.5ex]{2pt}{0.8pt}}}\ULon}

\makeatletter

\def\og{\leavevmode\raise.3ex\hbox{$\scriptscriptstyle\langle\!\langle$~}}
\def\fg{\leavevmode\raise.3ex\hbox{~$\!\scriptscriptstyle\,\rangle\!\rangle$}}

\newcommand{\wx}{{\overline x}}
\def\wG{ {\overline G}}

\def\wt{{\overline t}}

\def\wv{{\overline v}}
\def\wc{{\overline c}}

\def\bg{g_0}
\def\TT{{\mathbb{T}^N}}
\def\ZZ{{\mathbb{Z}^N}}
\def\pat{\partial_t}
\def\pa{\partial}

\def\fv{\widehat{v}}

\def\fv{\widehat{v}}

\def\fg{\widehat{G}}
\def\tn{\mathcal{N}}

\def\wtn{\widehat{\tn}}

\newcommand{\Vast}{\bBigg@{3}}
\newcommand{\vast}{\bBigg@{2}}
\newcommand{\vvast}{\bBigg@{1.5}}

\def\de{\delta}

\def\D{\Delta }
\def\vp{\varphi}
\def\lll{\langle}
\def\rrr{\rangle}

\def\al{\alpha}
\def\de{\delta}
\def\vp{\varphi}
\def\b{\beta}

\def\lm{\lambda}

\newcommand{\N}{\nabla}
\newcommand{\RR}{\mathbb{R}}

\newcommand{\ds}{\displaystyle}

\DeclareMathOperator{\R}{\mathbb{R}}
\def\qq{\qquad}
\def\q{\quad}

\definecolor{bor}{cmyk}{0.21,0.93,0.86,0.12}
\definecolor{air}{rgb}{0.178, 0.51, 0.51}
\definecolor{range}{cmyk}{0,0.599,1,0.188}

\def\ds{\displaystyle}

\newcommand{\pare}[1]{\left(#1\right)}

\newcommand{\set}[1]{\left\{#1\right\}}
\newcommand{\av}[1]{\left|#1\right|}
\newcommand{\norm}[1]{\|#1\|}
\newcommand{\w}[1]{\widehat{#1}}
\newcommand{\ov}[1]{\overline{#1}}
\renewcommand{\t}[1]{\text{#1}}

\def\cu{c_1}
\def\cd{c_2}
\def\deu{\de_1}
\def\ded{\de_2}

\author[M. Magliocca]{Martina Magliocca}

\address[M. Magliocca]{Department of Mathematical Analysis, 
Faculty of Mathematics, University of Sevilla,
C/Tarfia s/n, Campus Reina Mercedes,
Sevilla 41012, Spain.
\\ \url{mmagliocca@us.es}}


\begin{document}

\title{On a fourth order equation describing single-component film models}

\begin{abstract}
We study existence results for a fourth order problem describing  single-component film models assuming initial data in Wiener spaces.
\end{abstract}

\maketitle
\setcounter{tocdepth}{1}
\tableofcontents

\section{Introduction}\label{intro}

The aim of this work consists in analyzing, from a mathematical point of view, a specific physical model which describes controlled solid-state dewetting processes. The importance of this type of processes is due to their several applications.

We here focus on \emph{the single-component film with uniform composition of the bulk and the surface}.
The study of this particular problem has been inspired by the paper proposed by M. Khenner \cite{K}. In his work, M. Khenner developed a theoretical model which describes the  redistribution of the alloy components in the bulk of the film and on its surface during the dewetting process.
\\
A detailed physical description and some applications of this problems are
 proposed in  \cite[Sections II.A \& II.C]{K}. More precisely, the interested reader may find the derivation of this kind of problems in  \cite[Section II.A]{K}, where the binary alloy model considers two different atomic species, while  \cite[Section  II.C]{K} contains the $1$-dimensional evolution PDE (with only one atomic specie)  that we are going to study. \\
Moreover, M. Khenner underlined the link between controlled solid-state dewetting processes and   new technologies based on nonlinear optics, plasmonics, photovoltaics and photocatalysis.
Some concrete  examples are contained in the works of P. Heger et al \cite{1} and R. Santbergen et al \cite{3} as far as optics is concerned,  H. Liao et al \cite{2},  and S. Fafard et al \cite{4} regarding, respectively, nanomedicine and applied physic.

The problem in object is given by the following Mullins type equation:
\begin{equation}\label{K}
\ds\pat u=-V Mg_0u_{xxxx}-
V M\left( \left(-\frac{\cu  d}{d+u}+\frac{\cd  d^2}{(d+u)^2}\right) u_{xx}-\left(
\frac{\cu d}{(d+u)^2}-\frac{2\cd d^2}{(d+u)^3}
\right)  \right)_{xx},
\end{equation}
where $u=u(t,x)>0$ represents the height of the film surface. The complete list of the physical parameters involved in \eqref{K} is contained in \cite[Table $I$]{K} but, for the reader's convenience, we indicate here their most common values.\\
\begin{table}[h]
\renewcommand{\arraystretch}{1.5}
\begin{tabular}{c|c|c|c}
\makecell{\bf Physical \\ \bf parameter\\[2mm]} &\bf  Description &\bf  \makecell{\bf Typical values \\ \bf and  measure unit\\[2mm]} &\bf  Reference \\
\hline
$d$ & Length parameter & $0.5\times10^{-7} $ cm& \cite[C.-H. Chiu]{32}, \cite[Y. R. Niu et al]{34}\\
$V$ & Atomic volume & $10^{-22}$ cm$^3$&\cite[Table $I$]{K}\\
$\bg$ & Surface energy & $2.5\times 10^3$ erg/$\t{cm}^2$&\cite[H.L. Skriver et al]{38}\\
$c_i$ & (see below) & $0.025-2.5$ erg/$\t{cm}^2$&\cite[Z. Suo et al]{31}\\
$M$ & surface mobility   & $4\times 10^{-3}$\t{cm}$\times$ \t{s}/ $\t{g}$& \cite[Tables $I,\,II$]{K}
\end{tabular}
\vspace*{3mm}
\caption{Typical values of the parameters}\label{table}
\end{table}

The terms $-\cu d/(u+d)$ and $\cd d^2/(u+d)^2$ take into account the wetting interaction between the film surface and the substrate.
The term  $-\cu d/(u+d)$ (resp. $\cd d^2/(u+d)^2$) describes the long-range van der Waals-type attraction (resp. the  short-range repulsion). The constants $c_i$ are connected to the Hamaker constant $A_H$
 through the relation $c_i\sim A_H/12\pi d^2$.

As far as \eqref{K} is concerned, the main result provided by M. Khenner \cite[Section $II$.C]{K} is the \emph{linear stability analysis} of \eqref{K}. He thus  exploited a linearization argument about $u(t,x)=u(0,x)$, being $u(0,x)>0$ the film thickness immediately after the film deposition and at the beginning of the annealing phase. He proved that  a film of uniform thickness $u$ is stable (resp. unstable) if $u(0,x)>(<)d(3\cd /\cu -1)$.

\medskip

Our current goal is proving \emph{existence and regularity results} for equation of \eqref{K} type in higher dimensions.
The technique we are going to adopt is inspired to the one used in a work of  R. Granero-Belinchón with the author \cite{GM}. Here, we analysed several fourth order problems modelling the growth of crystal surfaces, among which
\begin{equation}\label{exp}
\pat u=\D e^{-\D u}\qq\t{in}\q (0,T)\times\TT.
\end{equation}
In short, we begin manipulating \eqref{exp} to obtain its quasilinear version
\begin{equation*}
\pat u=-\D^2u+\tn(t,x,u,\ldots,\D^2 u).
\end{equation*}
In this way, we have the harmonic heat equation plus some nonlinear terms depending also on $\D^2 u$. We thus pass to the Fourier formulation and then the proof reduces to proving a suitable a priori estimates and compactness results. \\
A similar approach is contained in the work of J.-G. Liu \& B. Strain \cite{LS}, where they set the equation \eqref{exp} in $\R^N$ and make use of the Fourier transform.\\
This technique works also with several problems arising in Fluid Dynamics. For instance, it has been previously employed to study different type of problems: for instance, we quote the work of  G. Bruell \& R. Granero-Belinchón regarding the evolution of thin films in Darcy and Stokes flows \cite{BGB}, by D. Córdoba and F. Gancedo \cite{CG}, P. Constantin, D. Córdoba, F. Gancedo, Rodriguez-Piazza, \& B. Strain \cite{CCGRS} for the Muskat problem (see also [6] and [16]), by Burczak \& R. Granero-Belinchón \cite{BuGB} to analyze the Keller-Segel system of PDE with diffusion given by a nonlocal operator and by H. Bae, R. Granero-Belinchón \& O. Lazar \cite{BGBL} to prove several global existence results (with infinite $L^p$ energy) for nonlocal transport equations.

\subsection{Notations and basic tools}

We set us in the torus $\TT=[-\pi,\pi]^N$, $N\ge 1$, and assume initial data $u_0>0$ in the Wiener space $A^0(\TT)$, being
\begin{align*}\label{Wienerhomo}
{{A}}^\alpha(\TT)
&=\left\{u(x)\in L^1(\TT):\q  \|u\|_{A^\alpha}:=\sum_{k\in\ZZ} |k|^\alpha|\widehat{u}(k)|<\infty\right\}.
\end{align*}
We also introduce the space of Radon measures from an interval $[0, T ]$ to a Banach
space $X$,
\[
\mathcal{M}(0,T;X).
\]

For a generic function $f(t,x)$, we will often write $f(x)$ (resp. $f(t)$) when referring to $f(\cdot,x)$ (resp. $f(t,\cdot)$).
We write
$$
f,_{j}=\frac{\partial f}{\partial x_j}\q j=1,\ldots, N,
$$
and we adopt Einstein convention for summation.

We name $c$ positive constants which may vary line to line during our incoming proofs.

We will make largely use of the interpolation inequality \cite[Lemma 2.1]{BGB}:
\begin{equation}\label{interpol}
\|u\|_{A^p}\le \|u\|_{A^0}^{1-\theta}\|u\|_{A^q}^{\theta}\q\t{for}\q 0\le p\le q,\,\, \theta=\frac{p}{q}.
\end{equation}

\subsection{The dimensionless problem}\label{app}

The aim of this Section is rewriting the equation \eqref{K} for $N\ge 1$ and in its dimensionless form.
We will represent with $\,\,\ov\cdot\,\,$ all the dimensionless terms involved.\\
Because of physical reasons, we assume $u(0,x)=u_0(x)>0$ in the whole paper.\\

Let
\begin{equation*} 
G(u)=-\frac{c_1d}{d+u}+\frac{c_2d^2}{(d+u)^2}.
\end{equation*}
Then, equation \eqref{K} in higher dimension reads
\begin{align*}
\ds\pat u&=-V Mg_0\D^2u-
V M\D\left(G(u)\D u- G'(u) \right).
\label{dim}
\end{align*}
We observe that
\[
\int_{\TT} \pat u\,dx=0,
\]
from which we deduce the mass conservation
\[
\int_{\TT} u(t,x)\,dx=\int_{\TT} u_0(x)\,dx,
\]
so it is natural to consider solutions of the type
\[
u(t,x)=\norm{u_0}_{L^1}+v(t,x)\q \t{with}\q \int_{\TT} v(t,x)\,dx=0.
\]

We  set the dimensionless variables and the unknown
\[
\wx=\frac{x}{d},\qq \wt=\frac{VMg_0}{d^4}t, \qq \wv(\wx,\wt)=\frac{v(t,x)}{\norm{u_0}_{L^1}+d}.
\]
We also define the dimensionless parameters
\[
 \wc_1=\frac{d}{d+\norm{u_0}_{L^1}}\frac{c_1}{\bg},\qq  \wc_2=\frac{d^2}{\pare{d+\norm{u_0}_{L^1}}^2}\frac{c_2}{\bg},
\]
and the function
\begin{equation*}
\wG(\wv)=-\frac{\wc_1}{1+\wv}+\frac{\wc_2}{(1+\wv)^2}.
\end{equation*}
Then the dimensionless equation (in compact form) reads
\begin{align}\label{compact}
 \pa_{\wt} \wv
&=-\D_{\wx}^2\wv-
\D_{\wx}\left(\wG(\wv)\D_{\wx} \wv- \wG'(\wv) \right).
\end{align}

\medskip

We want to recover a more handy version of \eqref{compact}. To this aim, we pass to the Einstein notation and compute
\begin{align}
 \pa_{\wt} \wv&=-\wv,_{iijj}-
\left(\wG(\wv) \wv,_{ii}- \wG'(\wv)  \right),_{jj}\nonumber\\
&=-\wv,_{iijj}-
\left(\wG(\wv)\wv,_{iijj}+\wG'(\wv)\left[2\wv,_j\wv,_{iij}+\wv,_{jj}\wv,_{ii}\right]+\wG''(\wv)\wv,_j\wv,_j\wv,_{ii}\right)
\nonumber
\\
&\q + \left(\wG''(\wv)\wv,_{jj}+\wG'''(\wv)\wv,_j\wv,_j\right).\label{main}
\end{align}
We want to isolate all the terms which involve both $\D\cdot$ and $\D^2\cdot$. We  take advantage of the following formula:
\begin{align*}
(1+y)^{-m}&=
\sum_{r\ge0} (-1)^r\binom{m+r-1}{m-1}y^r,
\end{align*}
in order to rewrite the $p$-th derivative of $\wG $ as
\begin{align*}
(\wG(\wv))^{(p)}
&=(-1)^pp!
\left[
-\frac{\wc_1}{(1+\wv)^{1+p}}
+(p+1)\frac{\wc_2 }{(1+\wv)^{2+p}}
\right]\\
&=(-1)^pp!
\left[
-\wc_1 \sum_{r\ge0} (-1)^r\binom{p+r}{p}\wv^r
+(p+1)\wc_2\sum_{r\ge0} (-1)^r\binom{p+r+1}{p+1}\wv^r
\right].
\end{align*}
In this way, the terms in \eqref{main} assume the form
\begin{align*}
\wG(\wv) \wv,_{iijj}&=(-\wc_1+\wc_2)\,\wv,_{iijj}+\sum_{r\ge1} (-1)^r\,
\w{\tn}_{0}^r,
\\
\wG'(\wv)\left[2\wv,_j\wv,_{iij}+\wv,_{jj}\wv,_{ii}\right]&=\sum_{r\ge0} (-1)^r\,\w{\tn}_{1}^r,
\\
\wG''(\wv)\wv,_j\wv,_j\wv,_{ii}&=\sum_{r\ge0} (-1)^r\,\w{\tn}_{2}^r,
\\
   \wG''(\wv)\wv,_{jj}&= (-\wc_1+3\wc_2)
\wv,_{jj}+ \sum_{r\ge1} (-1)^r\,\w{\tn}_{3}^r,
\\
 \wG'''(\wv)\wv,_j\wv,_j&= \sum_{r\ge0} (-1)^r\,\w{\tn}_{4}^r,
\end{align*}
being
\begin{align*}
\ov{\tn}_{0}^r&=
[-\wc_1+(r+1)\wc_2]\,\wv^r\wv,_{iijj},
\\
\ov{\tn}_{1}^r&=
[\wc_1(r+1)-2\wc_2(r+2)(r+1)]\,\wv^r
\,(2\wv,_j\wv,_{iij}+\wv,_{jj}\wv,_{ii}),
\\
\ov{\tn}_{2}^r&= 2
[-\wc_1(r+2)(r+1)+3\wc_2(r+3)(r+2)(r+1)]
\,\wv^r
\,\wv,_j\wv,_j\wv,_{ii} ,\\
\ov{\tn}_{3}^r&=2 [-\wc_1(r+2)(r+1)+3\wc_2(r+3)(r+2)(r+1)]\,
\wv^r
\wv,_{jj},
\\
\ov{\tn}_{4}^r&=3! [\wc_1(r+3)(r+2)(r+1)-4\wc_2(r+4)(r+3)(r+2)(r+1)]\,
\wv^r
\wv,_j\wv,_j.
\end{align*}
We conclude this algebraic part rewriting the equation \eqref{main} in the following way:
\begin{align*}
\ds  \pa_{\wt} \wv&=-(1+\wc_2-\wc_1)\wv,_{iijj}-2 \pare{\wc_1-3\wc_2}\,
\wv,_{jj}\nonumber\\
&\q-\sum_{r\ge1} (-1)^r\,
\pare{\ov{\tn}_{0}^r-\ov{\tn}_{3}^r}-
\sum_{r\ge0} (-1)^r\pare{\ov{\tn}_{1}^r+\ov{\tn}_{2}^r-\ov{\tn}_{4}^r} .
\end{align*}
Note that the nonlinear character of \eqref{compact} is contained only in the $\ov{\tn}^r_i$ terms.

\medskip

\begin{remark}
We observe that, unlike in \cite{GM}, equation \eqref{compact} is not invariant under the scaling $(x,t)\to(\lm^4 x,t)$, and this is due to the presence of lower order terms which act as perturbations. However, this scaling is still invariant for the $A^0$ semi-norm.
\end{remark}

\section{Main results}
 In what follows, we omit the $\,\,\ov{\cdot}\,\,$ superscript.\\
We now focus on the dimensionless problem
\begin{equation}\label{pb}\tag{P}
\begin{cases}
\begin{array}{ll}
\ds\pat v
=-\D^2v-
\D\left( G(v)\D v-  G'(v) \right)  & \text{in}\q  (0,T)\times\TT, \\
\ds v(0,x)=v_0(x) &\t{in}\q \TT,
\end{array}
\end{cases}
\end{equation}
being $G:(0,\infty)\to\RR$ continuous and defined as
\begin{equation}\label{G}
G(v)=-\frac{\cu }{1+v}+\frac{\cd }{(1+v)^2},\qq c_i>0.
\end{equation}

The notion of solution we are going to consider is given below.
\begin{definition}\label{defsol}
We say that a function $v$ is a weak solution of \eqref{pb} if
\[
v\in L^\infty( (0,T)\times\TT)\cap L^2(0,T;H^2(\TT))
\]
and verifies the following weak formulation:
\[
-\int_\TT v_0\vp(0)\,dx+\iint_{\TT\times(0,T)}- v\pat\vp+v\D^2\vp+\left(  G(v)\D v-   G'(v) \right)\D\vp\,dx\,dt=0
\]
for every $\vp\in W^{1,1}(0,T;L^1(\TT))\cap L^1(0,T;W^{4,1}(\TT))\cap L^2(0,T;H^2(\TT))$.
\end{definition}

\medskip
We begin presenting our existence results.

\begin{theorem}[Existence and regularity results with $A^0$ data]\label{teoex}
Assume that the parameters $\cu$, $\cd$ verify
\begin{equation}\label{poscost}
\cd>\max\set{c_1-1,\frac{\cu}{3}}.
\end{equation}
Let   $v_0\in A^0(\TT)$ such that
\[
\norm{v_0}_{A^0}<1.
\]
Defined the values $\de_i=\de_i(\norm{v_0}_{A^0})$  as
{\begin{align}
\deu &=\norm{v_0}_{A^0}
\left[\frac{\cu }{1-\norm{v_0}_{A^0}}+
\frac{\cd (2-\norm{v_0}_{A^0})+3\cu
}{(1-\norm{v_0}_{A^0})^2}
+2\frac{3\cd
+\cu \norm{v_0}_{A^0}}{(1-\norm{v_0}_{A^0})^3}+\frac{3!\cd \norm{v_0}_{A^0}}{(1-\norm{v_0}_{A^0})^4}\right],
\label{de1}\\
\ded &=
\norm{v_0}_{A^0}
\left[
2\cu \frac{3-3\norm{v_0}_{A^0}+\norm{v_0}_{A^0}^2}{(1-\norm{v_0}_{A^0})^3}\right.\nonumber\\
&\q\qq\qq\left.+3!
\frac{\cd (4-6\norm{v_0}_{A^0}+4\norm{v_0}_{A^0}^2-\norm{v_0}_{A^0}^3)+\cu }{(1-\norm{v_0}_{A^0})^4}
+\frac{4!\cd }{(1-\norm{v_0}_{A^0})^5}
\right],
\label{de2}
\end{align}}
we also require $\norm{v_0}_{A^0}$ small enough in order to have
\begin{equation}\label{smallness}
D_1=1+\cd -\cu -\deu >0\q\t{and}\q D_2=2 (3\cd-\cu)-\ded >0.
\end{equation}
Then, there exist at least one global weak solution to equation \eqref{pb} such that
\begin{equation}\label{ex1}
v\in  L^\infty( (0,T)\times \TT)\cap  \mathcal{M}(0,T;W^{4,\infty}(\TT))\cap L^2(0,T;{H}^2(\TT))
\end{equation}
and
\begin{equation}\label{ex3}
v\in  L^{\frac{4}{r}}(0,T;W^{r,\infty}(\TT))\cap L^{4}(0,T;W^{1,\infty}(\TT))\q\text{for}\q r=1,2,3,
\end{equation}
for any $T>0$. Furthermore, we also have the continuity regularity
\begin{equation}\label{ex2}
v\in {C}([0,T];L^2(\TT)),
\end{equation}
and the exponential decay
\begin{equation}\label{dec1A0}
\|v(t)\|_{L^\infty}\leq \norm{v_0}_{A^0}\exp\pare{-( D_1+ D_2)t}\qq\forall t\in (0,T).
\end{equation}
\end{theorem}

\medskip

Note that the function $G $ defined in \eqref{G} is a continuous bounded function for $v\ge 0$ and it may be negative. In the following result we will need $1+G(v)$ to be strictly positive.  It is for this reason that we modify \eqref{poscost} as \eqref{poscost-bis}.

\medskip

\begin{theorem}[Regularity results with $A^0(\TT)\cap H^2(\TT)$ data]\label{teoreg}
Assume that the parameters $c_i$ verify
\begin{equation}\label{poscost-bis}
c_2>\max\set{c_1-1,\frac{c_1}{3},\frac{c_1^2}{4}}.
\end{equation}
Let $v_0\in A^0(\TT)\cap H^2(\TT)$ satisfy
\[
\norm{v_0}_{A^0}<1,
\]
and the smallness condition in \eqref{smallness} (see also\eqref{de1}, \eqref{de2}). Then, the solution also verifies
\begin{equation}\label{reg1}
v\in {C}([0,T];{H}^2(\TT))\cap L^2(0,T;{H}^4(\TT)),
\end{equation}
\begin{equation}\label{H2H4}
\|v(t)\|_{H^2}+C_1\int_0^T \|v(t)\|_{H^4}^2\,dt\leq C_2,
\end{equation}
where $C_1=C_1(\cu ,\,\cd )$ and $C_2=C_2(\norm{v_0}_{A^0})$.
\end{theorem}

\section{Existence and regularity results with  $v_0\in A^0(\TT)$}

In order to prove our existence results, we will make use of the following approximating problem
\begin{equation}\label{appr}\tag{$\t{P}_n$}
\begin{cases}
\begin{array}{ll}
\ds \pat v_n=
-\D^2v_n-
\D\left( G_n(v_n)\D v_n-  G_n'(v_n) \right) & \text{in}\q  (0,T)\times\TT, \\
\ds v_n(0,x)=v_0(x) &\t{in}\q \TT,
\end{array}
\end{cases}
\end{equation}
with
\begin{equation}\label{eq:Gn}
G_n(v_n)=-c_1\sum_{r=0}^{n}(-1)^rv_n^r+c_2\sum_{r=0}^{n}(-1)^r(r+1)v_n^r.
\end{equation}

\begin{remark}[On the approximating problem]
The solutions $\{v_n\}_n$ to \eqref{appr} can be constructed through the Faedo-Galerkin method. \\
Another possibility, is using a Rothe type approximation scheme  (see, for instance, \cite[Section 2]{W}, \cite[Theorem 4.3]{SW}) whose main ingredient is discretizing in the time variable. The next steps are a fixed point argument to prove the existence of solutions of this new elliptic problem, and 
 finding suitable estimates which do not depend on the discretization parameter  to recover the time variable.\\
The approximating solutions $\{v_n\}_n$ belong to $C^1(0,T;A^0(\TT))\cap L^1(0,T;A^4(\TT))$. The rigorous proof of this claim can be done  by following these steps. \\ We consider the regularized problem
\[
\begin{cases}
\begin{array}{ll}
\ds \pat v_n^{(\kappa)}=
-\mathscr{J}_{\kappa}*\D^2\mathscr{J}_{\kappa}*v_n^{(\kappa)}-
\mathscr{J}_{\kappa}*\D\left( G_n(v_n^{(\kappa)})\D\mathscr{J}_{\kappa}* v_n^{(\kappa)}-  G_n'(v_n^{(\kappa)}) \right) & \text{in}\q  (0,T)\times\TT, \\
\ds v_n^{(\kappa)}(0,x)=\mathscr{J}_{\kappa}*v_0(x) &\t{in}\q \TT,
\end{array}
\end{cases}
\]
being $\mathscr{J}_{\kappa}$  mollifiers. Then, following the arguments in \cite[Chapter 3]{MB}, we obtain that  $v_n^{(\kappa)}$ belongs to $C^1(0,T;A^0(\TT))$. 
It can be proved that $v_n^{(\kappa)}\in L^1(0,T;A^4(\TT))$ reasoning as in Proposition \ref{prioriw}. 
Now, we only need  a suitable $\kappa$-uniform estimates   to take the limit in this parameter, and then deducing that $\{v_n\}_n\in C^1(0,T;A^0(\TT))\cap L^1(0,T;A^4(\TT))$.
\end{remark}

\medskip

\begin{proposition}[A priori estimates in Wiener spaces]\label{prioriw}
Let the parameters $c_i$ verify \eqref{poscost}.
Assume that $v_0\in A^0(\TT)$ satisfies
\[
\norm{v_0}_{A^0}<1,
\]
and the smallness condition in \eqref{smallness} (see also \eqref{de1}, \eqref{de2}).
Then, every sequence $\{v_n\}_n$ of solutions of \eqref{appr} is uniformly bounded in
\begin{equation}\label{unifbouwiener}
W^{1,1}(0,T;A^0(\TT))\cap L^1(0,T;A^4(\TT)).
\end{equation}
Furthermore,
\begin{equation}\label{decun}
\|v_n(t)\|_{A^0}\le \exp\pare{
-\pare{ D_1+ D_2}t
} \|v_0\|_{A^0},
\end{equation}
for $D_i$ defined in \eqref{wdeu}--\eqref{wded}.
\end{proposition}
\begin{proof}
We omit  the subscript $n$, so we write $v$ when referring to $v_n$.\\
We are going to work with the quasilinear version of the main equation in \eqref{appr}, i.e.
\begin{align}
\ds  \pat v&=-(1+\cd -\cu )v,_{iijj}+2 (3\cd-\cu)\,
v,_{jj}\nonumber\\
&\q-\sum_{r=1}^n (-1)^r\,
\pare{{\tn}_{0}^r-{\tn}_{3}^r}-
\sum_{r=0}^n (-1)^r\pare{{\tn}_{1}^r+{\tn}_{2}^r-{\tn}_{4}^r} ,\label{eqqq}
\end{align}
for
\begin{align*}
\tn_{0}^r&=
[-\cu +(r+1)\cd ]\,v^rv,_{iijj},
\\
\tn_{1}^r&=
[\cu (r+1)-\cd (r+2)(r+1)]\,v^r
\,(2v,_jv,_{iij}+v,_{jj}v,_{ii}),
\\
\tn_{2}^r&=
[-\cu (r+2)(r+1)+\cd (r+3)(r+2)(r+1)]
\,v^r
\,v,_jv,_jv,_{ii} ,\\
\tn_{3}^r&= [-\cu (r+2)(r+1)+\cd (r+3)(r+2)(r+1)]\,
v^r
v,_{jj},
\\
\tn_{4}^r&= [\cu (r+3)(r+2)(r+1)-\cd (r+4)(r+3)(r+2)(r+1)]\,
v^r
v,_jv,_j.
\end{align*}

\medskip

\noindent
{\it Step $1$: The Fourier series of \eqref{eqqq}.}\\
We begin computing the Fourier series of \eqref{eqqq}. Read in terms of the $k$--th Fourier coefficient, we have:
\begin{align*}
\ds  \pat\fv&=-(1+\cd -\cu )|k|^4\fv-2 (3\cd-\cu)|k|^2\fv\nonumber\\
&\q-\sum_{r=1}^n (-1)^r\,
\pare{\wtn_{0}^r-\wtn_{3}^r}-
\sum_{r=0}^n (-1)^r \pare{\wtn_{1}^r+\wtn_{2}^r-\wtn_{4}^r},
\end{align*}
for
\begin{align*}
\wtn_0^r(k)&=
[-\cu +(r+1)\cd ]
\sum_{\al^1\in \ZZ}\ldots\sum_{\al^{r}\in \ZZ}
\fv(\al^r)
\prod_{s=1}^{r-1}\fv(\al^s-\al^{s+1})\,|k-\al^1|^4\fv(k-\al^1),\\
\wtn_1^r(k)&=
[\cu (r+1)-\cd (r+2)(r+1)]\sum_{\al^1\in \ZZ}\ldots\sum_{\al^{r+1}\in \ZZ}
\fv(\al^{r+1})\prod_{s=2}^{r}\fv(\al^s-\al^{s+1})\nonumber
\\
&\q\times\left[ 2\, (k_j-\al_j^1)\fv(k-\al^1)|\al^1-\al^2|^2(\al^1_j-\al^2_j)\fv(\al^1-\al^2)\right.\nonumber\\
&\qq\left.+|k-\al^1|^2\fv(k-\al^1)|\al^1-\al^2|^2\fv(\al^1-\al^2)\right],
\\
\wtn_{2}^r(k)&=
[-\cu (r+2)(r+1)+\cd (r+3)(r+2)(r+1)] \sum_{\al^1\in \ZZ}\ldots\sum_{\al^{r+2}\in \ZZ}
\fv(\al^{r+2})\nonumber\\
&\q\times\prod_{s=3}^{r+1}\fv(\al^s-\al^{s+1})\,(k_j-\al^1_j)\fv(k-\al^1) (\al^1_j-\al^2_j)\fv(\al^1-\al^2) |\al^2-\al^3|^2\fv(\al^2-\al^3),\\
\\
\wtn_3(k)&=
[-\cu (r+2)(r+1)+\cd (r+3)(r+2)(r+1)]\,\nonumber\\
&\q\times
\sum_{\al^1\in \ZZ}\ldots\sum_{\al^{r}\in \ZZ}
\fv(\al^r)
\prod_{s=1}^{r-1}\fv(\al^s-\al^{s+1})\,|k-\al^1|^2\fv(k-\al^1),\\
\wtn_4(k)&=
[\cu (r+3)(r+2)(r+1)-\cd (r+4)(r+3)(r+2)(r+1)] \sum_{\al^1\in \ZZ}\ldots\sum_{\al^{r+1}\in \ZZ}
\fv(\al^{r+1})\nonumber
\\
&\q\times
\prod_{s=2}^{r}\fv(\al^s-\al^{s+1})(k_j-\al^1_j)\fv(k-\al^1)(\al^1_j-\al^2_j)\fv(\al^1-\al^2).
\end{align*}

\medskip

\noindent
{\it Step $2$: A priori estimates in $ L^1(0,T;A^4(\TT))\cap L^\infty(0,T;A^0(\TT))$ .}\\
Since
\[
\pat|\fv(t,k)|={Re\left( \overline{\fv}(t,k)\pat \fv(t,k) \right)}/{|\fv(t,k)|},
\]
  we  estimate the time derivative of the $A^0(\TT)$ semi-norm as
\begin{align}\label{disA0}
\ds \frac{d}{dt}\|v(t)\|_{A^0}&\le-(1+\cd -\cu )\|v(t)\|_{A^4}-2 (3\cd-\cu)\|v(t)\|_{A^2}
+\sum_{r\ge1}\pare{ \|\tn_{0}^r(t)\|_{A^0}+ \|\tn_{3}^r(t)\|_{A^0}}\nonumber\\
&\q+\sum_{r\ge0}
\pare{ \|\tn_{1}^r(t)\|_{A^0}+ \|\tn_{2}^r(t)\|_{A^0}+ \|\tn_{4}^r(t)\|_{A^0}}.
\end{align}
We make use of Tonelli's Theorem and the interpolation inequality \eqref{interpol}, getting
\begin{align*}
\|\tn_0^r(t)\|_{A^0}&\le [\cu +(r+1)\cd ]{\|v(t)\|_{A^0}^r} \|v(t)\|_{A^4}
,\\
\|\tn_1^r(t)\|_{A^0}&\le
 [\cu (r+1)+\cd (r+2)(r+1)]{\|v(t)\|_{A^0}^r}\left[ 2 \|v(t)\|_{A^1}\|v(t)\|_{A^3}+  \|v(t)\|_{A^2}^2\right]
\\
&\le
3[\cu (r+1)+\cd (r+2)(r+1)]
{\|v(t)\|_{A^0}^{r+1}}  \|v(t)\|_{A^4}
,\\
\|\tn_{2}^r(t)\|_{A^0}&\le
[\cu (r+2)(r+1)+\cd (r+3)(r+2)(r+1)]{\|v(t)\|_{A^0}^r}\|v(t)\|_{A^1}^2\|v(t)\|_{A^2}\\
&\le
[\cu (r+2)(r+1)+\cd (r+3)(r+2)(r+1)]{\|v(t)\|_{A^0}^{r+2}}\|v(t)\|_{A^4};
\\
 \|\tn_3^r(t)\|_{A^0}&\le [\cu (r+2)(r+1)+\cd (r+3)(r+2)(r+1)]\|v(t)\|_{A^0}^r\|v(t)\|_{A^2}, \\
 \|\tn_4^r(t)\|_{A^0}&\le [\cu (r+3)(r+2)(r+1)+\cd (r+4)(r+3)(r+2)(r+1)]\|v(t)\|_{A^0}^r\|v(t)\|_{A^1}^2
\\
&\le [\cu (r+3)(r+2)(r+1)+\cd (r+4)(r+3)(r+2)(r+1)]\|v(t)\|_{A^0}^{r+1}\|v(t)\|_{A^2}
.
\end{align*}
We now deal with the power series in $r$. To this aim, we recall that, for any $0<w<1$, it holds that
\begin{align*}
\sum_{r\ge0} w^r&=\frac{1}{1-w},\\
\sum_{r\ge0}\prod_{j=1}^m (r+j)w^{r}&=\left(
\sum_{r\ge0} w^{r+m}
\right)^{(m)}=\left(\frac{w^m}{1-w}\right)^{(m)}=\frac{m!}{(1-w)^{m+1}},\\
\sum_{r\ge1}\prod_{j=1}^m (r+j)w^{r}&=\left(\frac{w^{m+1}}{1-w}\right)^{(m)}=m!\frac{1-(1-w)^{m+1}}{(1-w)^{m+1}}.
\end{align*}
We use the above computations with $w=\|v(t)\|_{A^0}$ in order to  bound the sums of $ \|\tn_i^r(t)\|_{A^0}$ as
\begin{align*}
\sum_{r\ge1} \|\tn_0^r(t)\|_{A^0}&\le\|v(t)\|_{A^0}
\left(
\frac{\cu }{1-\|v(t)\|_{A^0}}+
\frac{\cd \left(2-\|v(t)\|_{A^0}\right)}{(1-\|v(t)\|_{A^0})^2}
\right)
\|v(t)\|_{A^4}
,\\
\sum_{r\ge0} \|\tn_1^r(t)\|_{A^0}&\le
3\|v(t)\|_{A^0}
\left(
\frac{\cu }{(1-\|v(t)\|_{A^0})^2}+
\frac{2\cd }{(1-\|v(t)\|_{A^0})^3}
\right)
 \|v(t)\|_{A^4},\\
\sum_{r\ge0} \|\tn_{2}^r(t)\|_{A^0}&\le
\|v(t)\|_{A^0}^2
\left(
\frac{2\cu }{(1-\|v(t)\|_{A^0})^3}+
\frac{3!\cd }{(1-\|v(t)\|_{A^0})^4}
\right)
   \|v(t)\|_{A^4};\\
\sum_{r\ge1} \|\tn_3^r(t)\|_{A^0}&\le  \norm{v(t)}_{A^0}
\left(2\cu
\frac{3-3\norm{v(t)}_{A^0}+\norm{v(t)}_{A^0}^2}{(1-\norm{v(t)}_{A^0})^3}\right.\\
&\qq\qq\qq\left.+
3!\cd \frac{4-6\norm{v(t)}_{A^0}+4\norm{v(t)}_{A^0}^2-\norm{v(t)}_{A^0}^3}{(1-\norm{v(t)}_{A^0})^4}
\right)
\norm{v(t)}_{A^2}
,\\
\sum_{r\ge0} \|\tn_4^r(t)\|_{A^0}&\le  \|v(t)\|_{A^0}
\left(
\frac{3!\cu }{(1-\|v(t)\|_{A^0})^4}
+\frac{4!\cd }{(1-\|v(t)\|_{A^0})^5}
\right) \|v(t)\|_{A^2}.
\end{align*}
We gather the above estimates in two groups w.r.t. the $A^4(\TT)$ and the $A^2(\TT)$ semi-norms:
\begin{align*}
\sum_{r\ge1} \|\tn_0^r(t)\|_{A^0}+\sum_{\ell=1}^{2}\sum_{r\ge0} \|\tn_\ell^r(t)\|_{A^0}&\le
\|v(t)\|_{A^0}\left[
\frac{\cu }{1-\|v(t)\|_{A^0}}+
\frac{\cd (2-\|v(t)\|_{A^0})}{(1-\|v(t)\|_{A^0})^2}
\right. \\
&\q+
3
\left(
\frac{\cu }{(1-\|v(t)\|_{A^0})^2}+
\frac{2\cd }{(1-\|v(t)\|_{A^0})^3}
\right)
\\
&\q\left.+{\|v(t)\|_{A^0}}
\left(
\frac{2\cu }{(1-\|v(t)\|_{A^0})^3}+
\frac{3!\cd }{(1-\|v(t)\|_{A^0})^4}
\right)
 \right]  \|v(t)\|_{A^4}\\
&=\deu (\|v(t)\|_{A^0})\|v(t)\|_{A^4},\\
\sum_{r\ge1} \|\tn_3^r(t)\|_{A^0}+\sum_{r\ge0} \|\tn_4^r(t)\|_{A^0}
\le&
\|v(t)\|_{A^0}
\left(
2\cu \frac{3-3\norm{v(t)}_{A^0}+\norm{v(t)}_{A^0}^2}{(1-\norm{v(t)}_{A^0})^3}\right.\\
&\qq\left.+3!
\frac{\cd (4-6\norm{v(t)}_{A^0}+4\norm{v(t)}_{A^0}^2-\norm{v(t)}_{A^0}^3)+\cu }{(1-\norm{v(t)}_{A^0})^4}
\right.\\
&\qq\left.+\frac{4!\cd }{(1-\|v(t)\|_{A^0})^5}
\right)\norm{v(t)}_{A^2}
\\
&=\ded (\|v(t)\|_{A^0})\|v(t)\|_{A^2} ,
\end{align*}
where $\de_i (\|v(t)\|_{A^0})=\de_i (t)$ are
\begin{align*}
\deu (t)&=
\|v(t)\|_{A^0}
\pare{\frac{\cu }{1-\|v(t)\|_{A^0}}+
\frac{\cd (2-\|v(t)\|_{A^0})+3\cu
}{(1-\|v(t)\|_{A^0})^2}
+2\frac{3\cd
+\cu \|v(t)\|_{A^0}}{(1-\|v(t)\|_{A^0})^3}+\frac{3!\cd \|v(t)\|_{A^0}}{(1-\|v(t)\|_{A^0})^4}},\\
\ded (t)&=
\|v(t)\|_{A^0}
\left(
2\cu \frac{3-3\norm{v(t)}_{A^0}+\norm{v(t)}_{A^0}^2}{(1-\norm{v(t)}_{A^0})^3}\right.\\
&\qq\qq\qq\left.+3!
\frac{\cd (4-6\norm{v(t)}_{A^0}+4\norm{v(t)}_{A^0}^2-\norm{v(t)}_{A^0}^3)+\cu }{(1-\norm{v(t)}_{A^0})^4}
+\frac{4!\cd }{(1-\|v(t)\|_{A^0})^5}
\right).
\end{align*}
We come back to the main inequality \eqref{disA0} so that, thanks to the above computations, we obtain
\begin{align*}
 \frac{d}{dt}\|v(t)\|_{A^0}&\le-(1+\cd -\cu -\deu (t))\|v(t)\|_{A^4}-\pare{2 (3\cd-\cu)-\ded (t)}\|v(t)\|_{A^2}.
\end{align*}
We are going to prove that, thanks to  the smallness condition \eqref{smallness}, we also have  that
\begin{align*}
1+\cd -\cu -\deu (t)&>0,\\
2 (3\cd-\cu)-\ded (t)&>0,
\end{align*}
for small values of $t$. Indeed, since $\de_i(t)=\de_i(\|v(t)\|_{A^0} )$ are  increasing functions in $\|v(t)\|_{A^0}$, then
\[
\de_i(t)\le \de_i(0)=\de_i\q\forall t\le t^*,
\]
where
\[
t^*=\sup\left\{ \tau:\q  \|v(t)\|_{A^0}\le \|v_0\|_{A^0}\q\forall t\le \tau \right\}.
\]
This implies that, for $t\le t^*$,
\begin{align}
1+\cd -\cu -\deu (t)\ge1+\cd -\cu -\deu =D_1&>0,\label{wdeu}\\
2 (-\cu +3\cd )-\ded (t)\ge2 (-\cu +3\cd )-\ded =D_2&>0.\label{wded}
\end{align}
Then we have
\begin{equation}\label{disA4}
 \frac{d}{dt}\|v(t)\|_{A^0}+D_2\|v(t)\|_{A^2}+D_1\|v(t)\|_{A^4}\le 0\qq t\le t^*,
\end{equation}
so we deduce that $v\in L^1(0,t^*;A^4(\TT))\cap L^\infty(0,t^*;A^0(\TT))$.\\
We claim that this regularity holds for every time. If not, namely $t^*<\infty$, then we  would have that
\[
 \frac{d}{dt}\|v(t)\|_{A^0}\biggl|_{t=t^*}<0
\]
which leads to a contradiction.
\\
We conclude this step deducing that the  uniform boundedness
$v\in L^1(0,T;A^4(\TT))\cap L^\infty(0,T;A^0(\TT))$  holds up to $T\le \infty$.\\

\noindent
{\it Step $3$: The  regularity $ W^{1,1}(0,T;A^0(\TT))$.}\\
The continuity regularity follows from the inequality
\[
 \frac{d}{dt}\|v(t)\|_{A^0}\le c\pare{ D_2\|v(t)\|_{A^2}+ D_1\|v(t)\|_{A^4}}\qq  \forall t>0.
\]

\noindent
{\it Step $4$: The decay estimate.}
The inequality in \eqref{disA4} for $t\le \infty$,  the fact that $\norm{v}_{A^\al}\le \norm{v}_{A^\b}$ for $\alpha\le\b$ and Gronwall's inequality
provide us with the exponential decay
\begin{equation*}
\|v(t)\|_{A^0}\le \exp\pare{
-\pare{ D_1+ D_2}t
} \|v_0\|_{A^0}\qq  \forall t>0.
\end{equation*}

\end{proof}

\medskip

We collect the compactness results we need to prove the existence of solutions as \cref{defsol} in the following Proposition.

\medskip

\begin{proposition}[Compactness results]\label{compex}
Let the parameters $c_i$ verify \eqref{poscost}.
Assume that $v_0\in A^0(\TT)$ satisfies
\[
\norm{v_0}_{A^0}<1,
\]
and the smallness condition in \eqref{smallness} (see also \eqref{de1}, \eqref{de2}).  Then, there exists a function $v$ such that, up to subsequences,  every approximating sequence of solutions $\{v_n\}_n$ of \eqref{appr} verifies
\begin{align}
v_n\rightarrow v\q&\t{a.e.}\q (0,T)\times\TT,\label{ae}\\
v_n\overset{*}{\rightharpoonup} v\q& \t{in}\q L^\infty( (0,T)\times\TT)
,\label{Linf}\\
v_n\overset{*}{\rightharpoonup} v\q& \t{in}\q \mathcal{M}(0,T;W^{4,\infty}(\TT))
,\label{measw}\\
v_n\rightharpoonup v\q& \t{in}\q L^2(0,T;H^{2}(\TT))
.\label{LH}
\end{align}
\end{proposition}

\begin{proof}
We take advantage of \eqref{unifbouwiener} and the Banach-Alaoglu Theorem to deduce \eqref{Linf}--\eqref{measw}.

Since $\|f\|_{W^{\al,\infty}}\le \|f\|_{A^\al}$ for every $\al\ge0$, the weak convergence \eqref{LH} directly follows interpolating between $0<2<4$ (see the interpolation inequality in \eqref{interpol}) and by the uniform bounds in \eqref{unifbouwiener}:
\begin{equation}\label{finH2}
\int_0^T \|v_n(t)\|_{H^2}^{2}\,dt\le c\int_0^T \|v_n(t)\|_{A^2}^{2}\,dt\le c\|v_n(t)\|_{L^\infty(A^0)}\int_0^T\|v_n(t)\|_{A^4}\,dt<\infty.
\end{equation}
The a.e. convergence \eqref{ae} is a consequence of the above ones.
\end{proof}

\medskip

The following convergence result follows from \cref{prioriw,compex}, and it will allow us to complete the proof of \cref{teoex}.

\medskip

\begin{corollary}[Further regularity results]\label{corregw}
Let the parameters $c_i$ verify \eqref{poscost}.
Assume that $v_0\in A^0(\TT)$ satisfies
\[
\norm{v_0}_{A^0}<1,
\]
and the smallness condition in \eqref{smallness} (see also \eqref{de1}, \eqref{de2}).
 Then, up to subsequences, we  have that every approximating sequence of solutions $\{v_n\}_n$ of \eqref{appr} verifies
\begin{align}
v_n\overset{*}{\rightharpoonup} v\q& \t{in}\q L^{\frac{4}{r}}(0,T;W^{r,\infty}(\TT))\q\text{for}\q r=1,2,3
,\label{LW1}\\
v_n\rightarrow v\q& \t{in}\q L^2(0,T;H^{r}(\TT)),\q 0\le r<2.
\label{LHr}
\end{align}
Furthermore,
\begin{equation}\label{ct}
v\in C([0,T];L^2(\TT)).
\end{equation}
\end{corollary}
\begin{proof}
The interpolation inequality \eqref{interpol} applied with $0<r<4$ gives us
\[
\int_0^T \|v_n(t)\|_{A^r}^{\frac{4}{r}}\,dt\le \|v_n(t)\|_{L^\infty(A^0)}^{\frac{4}{r}-1}\int_0^T\|v_n(t)\|_{A^4}\,dt<\infty
\]
thanks, again, to \eqref{unifbouwiener}. Since $\|f\|_{W^{\al,\infty}}\le \|f\|_{A^\al}$ for every $\al\ge0$, we recover the $*$-weak convergence in \eqref{LW1}.

We now focus on \eqref{LHr}. To obtain this convergence, we claim that we only need to prove that
\begin{equation}\label{dt}
\pat v_n\rightharpoonup \pat v\q \t{in}\q L^2(0,T;H^{-2}(\TT)).
\end{equation}
Indeed, recalling \eqref{LH} and invoking classical compactness results (see, for instance, \cite[Corollary $4$]{S}), then we get
\begin{equation}\label{LL}
v_n\rightarrow v\q \t{in}\q L^2(0,T;L^{2}(\TT)),
\end{equation}
and this convergence leads to \eqref{LHr}, since  interpolation inequalities and \eqref{LH} imply
\[
\|v_n-v\|_{L^2(H^r)}^2\le \|v_n-v\|_{L^2(H^2)}^r \|v_n-v\|_{L^2(L^2)}^{2-r}<c\|v_n-v\|_{L^2(L^2)}^{2-r}.
\]
Then, we are left with the proof of \eqref{dt}. Thanks to the Riesz Theorem, we know that
\[
\|\pat v_n(t)\|_{H^{-2}}=\underset{\footnotesize
\begin{split}
&\vp\in H^2(\TT)\\
&\|\vp\|_{H^2}\le 1
\end{split}
}{\sup}\left|
\phantom{x}_{H^{-2}}\lll \pat v_n(t), \vp\rrr_{H^2}
\right|.
\]
Problem \eqref{appr} implies
\begin{align*}
 \left|
\int_\TT \pat v_n(t) \vp\,dx
\right|
&\le\left|
\int_\TT
\left(1 +G_n(v_n(t))\right)\D v_n(t)\D \vp\,dx\right|+ \left|\int_\TT G_n'(v_n(t))\D\vp\,dx
\right|\\
&\le  \left(1+\max_{v_n} |G_n(v_n)| \right) \|v_n(t)\|_{H^2}\|\vp\|_{H^2}+  \max_{v_n} |G_n'(v_n)|\|\vp\|_{H^2}\\
&\le \max\set{ 1+\max_{v_n} |G_n(v_n)|, \max_{v_n} |G_n'(v_n)| }\left(
1+ \|v_n(t)\|_{H^2}
\right)\|\vp\|_{H^2}.
\end{align*}
We claim that $G_n(v)$ and $G_n'(v)$ are bounded uniformly in $n$. We briefly detail the proof of the boundedness of $G_n(v)$, being the one of $G_n'(v)$ analogous. By definition of $G_n(v)$ in \eqref{eq:Gn}, we have that
\begin{align*}
\av{G_n(v)}
&\le c_1\sum_{r=0}^n\av{v}^{r}+c_2\sum_{r=0}^n(r+1)\av{v}^{r} \le \frac{c_1}{1-|v|}+c_2\pare{|v|\sum_{r=0}^n\av{v}^{r}}'. 
\end{align*}
We now recall that we have assumed $\norm{v_0}_{A^0}<1$ and also that, by \cref{prioriw} (Step 4), we have  $\norm{v(t)}_{A^0}\le\norm{v_0}_{A^0}<1$. Then,  we are allowed to bound $\av{G_n(v)}$ as
\[
\av{G_n(v)}\le \frac{c_1}{1-\norm{v_0}_{A^0}}+\frac{c_2}{(1-\norm{v_0}_{A^0})^2}<\infty.
\]
Reasoning in the same way, we deduce the boundedness of $\av{G_n'(v)}$.\\
We thus get that
\begin{equation}\label{findt}
	\|\pat v_n(t)\|_{H^{-2}}^2\le c\left(
1+ \|v_n(t)\|_{H^2}
\right)^2
\end{equation}
and the desired bound follow by \eqref{LH}.

Finally, the continuity regularity \eqref{ct} is due to \eqref{finH2}  and \eqref{findt}.
\end{proof}

\medskip

We are now ready to prove \cref{teoex}.

\begin{proof}[Proof of \cref{teoex}]
Let $\vp\in W^{1,1}(0,T;L^1(\TT))\cap L^1(0,T;W^{4,1}(\TT))\cap L^2(0,T;H^2(\TT))$. We begin proving that the limit in $n$ of
\[
- \int_\TT v_0\vp(0)\,dx+\iint_{(0,T)\times\TT}- v_n\pat\vp+  v_n\D^2\vp+\left(  G(v_n)\D v_n-  G'(v_n) \right)\D\vp\,dx\,dt=0
\]
gives us the weak formulation in \cref{defsol}. To this aim, we use the boundedness and convergence results obtained in \cref{prioriw,compex}. We have that the convergence
\[
\iint_{(0,T)\times\TT}- v_n\pat\vp+ v_n\D^2\vp\,dx\,dt
\to
\iint_{(0,T)\times\TT}- u\pat\vp+ u\D^2\vp\,dx\,dt
\]
holds thanks to \eqref{Linf} and since $\pat\vp,\,\D^2\vp\in L^1((0,T)\times\TT)$. Furthermore, since $G ,\,G' $ are continuous functions and $v_n\rightarrow v$ a.e., then
\[
\iint_{(0,T)\times\TT}\left(  G(v_n)\D v_n-  G'(v_n) \right)\D\vp\,dx\,dt
\to
\iint_{(0,T)\times\TT}\left(  G(v)\D u-  G'(v) \right)\D\vp\,dx\,dt
\]
being $v_n$ uniformly bounded in $L^2(0,T;H^2(\TT))$ and $\D\vp\in L^2((0,T)\times\TT)$.\\

The regularities in \eqref{ex1}, \eqref{ex3} and \eqref{ex2} follows combining \cref{compex,corregw}.\\

The exponential decay in \eqref{dec1A0} is due to the decay in \eqref{decun} and the weakly-*lower semicontinuity of the norm, which implies
\[
\|v(t)\|_{L^\infty}\le\lim_{n\to\infty}\|v_n(t)\|_{A^0}\le \exp\pare{
-\pare{ D_1+ D_2}t
} \|v_0\|_{A^0}.
\]
\end{proof}

\section{Regularity results with $v_0\in A^0(\TT)\cap H^2(\TT)$}

We first prove some inequalities we will use during the incoming section.
\begin{lemma}
We have that the following estimates hold:
\begin{align}
\|\N w\|_{L^2}^2&\le c\|w\|_{H^2}\|w\|_{A^0}
,\label{l2}\\
\|\N w\|_{L^4}^2&\le c\|w\|_{H^2}\|w\|_{A^0},\label{l4}\\
\|\N w\|_{L^4}^2&\le c\|w\|_{L^2}\|w\|_{A^2}.\label{l42}
\end{align}
\end{lemma}

\begin{proof}
We recall that  $\|f\|_{W^{\al,\infty}}\le \|f\|_{A^\al}$ for $\al\ge 0$, and proceed with the proofs integrating by parts.

We begin with \eqref{l2}:
\begin{align*}
\|\N w\|_{L^2}^2=\int_{\TT}  w,_jw,_j\,dx
&=-\int_{\TT} w,_{jj} w\,dx\le c\|w\|_{L^\infty}\|w\|_{H^2}\le
c\|w\|_{A^0}\|w\|_{H^2}.
\end{align*}

We now prove both the estimates  \eqref{l4}-\eqref{l42} regarding the $W^{1,4}(\TT)$ norm of $w$. Again, we integrate by parts obtaining
\begin{align*}
\|\N w\|_{L^4}^4&=\int_{\TT} w,_iw,_i w,_jw,_j\,dx =-\int_{\TT} (w,_iw,_i w,_j),_jw\,dx =-\int_{\TT} \pare{2w,_{ij}w,_iw,_j +w,_iw,_iw,_{jj}}w\,dx.
\end{align*}
Then, thanks also to H\"older's inequality, we estimate either as
\begin{align*}
&-\int_{\TT} \pare{2w,_{ij}w,_iw,_j +w,_iw,_iw,_{jj}}w\,dx \le c\|w\|_{A^0}\|w\|_{H^2}\|\N w\|_{L^4}^2,
\end{align*}
or
\begin{align*}
&-\int_{\TT} 2w,_{ij}\pare{w,_iw,_j +w,_iw,_iw,_{jj}}w\,dx\le
 c\|w\|_{L^2}\|w\|_{A^2}\|\N w\|_{L^4}^2,
\end{align*}
and the proof is concluded.
\end{proof}

\medskip

\begin{proposition}[A priori estimates in Sobolev spaces]\label{propreg}
Let the parameters $c_i$ verify \eqref{poscost}.
Assume that $v_0\in A^0(\TT)\cap H^2(\TT)$ satisfies
\[
\norm{v_0}_{A^0}<1,
\]
and the smallness condition in \eqref{smallness} (see also \eqref{de1}, \eqref{de2}).
Then, every approximating sequence of solutions $\{v_n\}_n$ of \eqref{appr} is uniformly bounded in
\[
\{v_n\}_n\in L^\infty(0,T;H^2(\TT))\cap L^2(0,T;{H}^4(\TT)).
\]
\end{proposition}
\begin{proof}
Again, we omit the $n$--subscript and the dependence of $v$ on $(t,\,x)$.

We multiply the equation in \eqref{appr} by $\D^2v$ and integrate in space, getting
\begin{align}\label{ddtD2}
\frac{1}{2}\frac{d}{dt}\int_{\TT}|\D v|^2\,dx&=-\int_{\TT}|\D^2 v|^2\,dx-\int_{\TT}\D\left(G_n(v)\D v-  G_n'(v)\right)\D^2v\,dx.
\end{align}
Using the Einstein notation, we split the last integral in \eqref{ddtD2} as
\begin{align*}
-\int_{\TT}\left(G_n(v) v,_{ii}\right),_{\ell\ell}v,_{jjkk}\,dx&=-\int_{\TT}G_n(v)v,_{ii\ell\ell}v,_{jjkk}\,dx\\
&\q
-\int_{\TT}\left(
G_n''(v)v,_{\ell}v,_{\ell}v,_{ii}+G_n'(v)v,_{\ell\ell}v,_{ii}\right)v,_{jjkk}\,dx\\
&\q
-2\int_{\TT}G_n'(v)v,_{\ell}v,_{ii\ell}v,_{jjkk}\,dx
,\\
\int_{\TT}(G_n'(v)),_{\ell\ell}v,_{jjkk}\,dx&=\int_{\TT}(
G_n'''(u)v,_{\ell}v,_{\ell}+G_n''(v)v,_{\ell\ell}
)v,_{jjkk}\,dx.
\end{align*}
Then,  setting
\begin{align*}
I_1&=-\int_{\TT}
\left(
G_n''(v)v,_{\ell}v,_{\ell}v,_{ii}+G_n'(v)v,_{\ell\ell}v,_{ii}
\right)v,_{jjkk}\,dx
,\\
I_2&=-2\int_{\TT}
G_n'(v)v,_{\ell}v,_{ii\ell}v,_{jjkk}\,dx
,\\
I_3&= \int_{\TT}(
G_n'''(v)v,_{\ell}v,_{\ell}+G_n''(v)v,_{\ell\ell}
)v,_{jjkk}\,dx,
\end{align*}
we rearrange  \eqref{ddtD2} as
\begin{align}
\frac{1}{2}\frac{d}{dt}\int_{\TT}|v,_{jj}|^2\,dx
=-\int_{\TT}(1+G_n(v))|v,_{jjkk}|^2\,dx+I_1+I_2+I_3.\label{star}
\end{align}
We now focus on the first integral in the r.h.s..  \\
We claim that, for $n$ large, we have
\[
1+ G_n(v)>0.
\]
Since $G_n(v)\to G(v)$ for $n\to \infty$, it is sufficient to check whether $1+G(v)>0$.
Computing $G'(v)$, we find that $G(v)$ attains a global minimum in $v=2c_2/c_1-1$, and this leads to
\[
1+\min_v G(v)=1-\frac{\cu^2}{4\cd}>0,\q\t{i.e.}\q c_2> \frac{\cu^2}{4}
\]
which, jointly with \eqref{poscost}, gives assumption \eqref{poscost-bis}.
\\
Then, thanks to the previous remarks, we have
\begin{align}
\frac{1}{2}\frac{d}{dt}\int_{\TT}|v,_{jj}|^2\,dx
&\le \av{I_1}+I_2+\av{I_3}.\label{abc}
\end{align}
In the following computations, we need $G_n(v)$, and its derivatives as well,  to be bounded. We omit this boundedness result and we refer to the proof of \cref{corregw} for more details.
\\

We begin estimating the $I_1$ term using the $L^4$-gradient estimate in \eqref{l4} with $w=v$, getting
\begin{align}
\av{I_1}&\le c\pare{\int_{\TT} |v,_{\ell}||v,_{\ell}||v,_{ii}||v,_{jjkk}|\,dx+\int_{\TT}|v,_{\ell\ell}||v,_{ii}||v,_{jjkk}|\,dx}\nonumber\\
&\le c\pare{\|\D^2v\|_{A^0}\|\N v\|_{L^4}^2\|v\|_{H^2}+\|\D^2v\|_{A^0}\|v\|_{H^2}^2}\nonumber\\
&\le c\pare{\|v\|_{A^4}\|v\|_{H^2}^2\|v\|_{A^0}+\|v\|_{A^4}\|v\|_{H^2}^2}\nonumber\\
&\le c(1+\|v\|_{A^0})\|v\|_{A^4}\|v\|_{H^2}^2.\label{estimateA}
\end{align}
As far as $I_3$ is concerned, we use \eqref{l2} with $w=v$ and estimate as
\begin{align*}
\av{I_3}&\le c\pare{\int_{\TT}
|v,_{\ell}||v,_{\ell}||v,_{jjkk}|\,dx+\int_{\TT}|v,_{\ell\ell}||v,_{jjkk}|\,dx}\\
&\le c\pare{\|\D^2v\|_{A^0}\|\N v\|_{L^2}^2+\|\D^2v\|_{A^0}\|v\|_{H^2}}\\
&\le c\pare{\|v\|_{A^4}\|v\|_{H^2}\|v\|_{A^0}+\|v\|_{A^4}\|v\|_{H^2}}\\
&\le c(1+\|v\|_{A^0})\|v\|_{A^4}\|v\|_{H^2}.
\end{align*}
We now deal with the $I_2$ term. We want an estimate of the type \eqref{estimateA}. In order to do so, we proceed with a chain of integration by parts aimed at rewriting $I_2$ in a handier form. We have
\begin{align*}
I_2&=\int_{\TT}G_n'(v)v,_{\ell}v,_{ii\ell}v,_{jjkk}\,dx\\
&=-\int_{\TT}(G_n'(v)v,_{\ell}v,_{ii\ell}),_kv,_{jjk}\,dx\\
&=-\int_{\TT}G_n''(v)v,_kv,_{\ell}v,_{ii\ell}v,_{jjk}\,dx
-\int_{\TT}G_n'(v)v,_{\ell k}v,_{ii\ell}v,_{jjk}\,dx -\int_{\TT}G_n'(v)v,_{\ell}v,_{ii\ell k}v,_{jjk}\,dx.
\end{align*}
We follow dealing with the last integral in the above equality:
\begin{align*}
-\int_{\TT}G_n'(v)v,_{\ell}v,_{ii\ell k}v,_{jjk}\,dx
&=\int_{\TT}(G_n'(v)v,_{\ell}v,_{jjk}),_iv,_{i\ell k}\,dx\\
&=\int_{\TT}G_n''(v)v,_iv,_{\ell}v,_{jjk}v,_{i\ell k}\,dx
+\int_{\TT}G_n'(v)v,_{\ell i}v,_{jjk}v,_{i\ell k}\,dx\\
&\q
+\int_{\TT}G_n'(v)v,_{\ell}v,_{jjki}v,_{i\ell k}\,dx.
\end{align*}
We are in the same situation as before. Then
\begin{align*}
\int_{\TT}G_n'(v)v,_{\ell}v,_{jjki}v,_{i\ell k}\,dx=
&-\int_{\TT}(G_n'(v)v,_{\ell}v,_{i\ell k}),_jv,_{jki}\,dx\\
=&-\int_{\TT}G_n''(v)v,_jv,_{\ell}v,_{i\ell k}v,_{jki}\,dx
-\int_{\TT}G_n'(v)v,_{\ell j}v,_{i\ell k}v,_{jki}\,dx\\
&\q
-\int_{\TT}G_n'(v)v,_{\ell}v,_{i\ell kj}v,_{jki}\,dx.
\end{align*}
We finally can say that
\begin{align*}
-\int_{\TT}G_n'(v)v,_{\ell}v,_{i\ell kj}v,_{jki}\,dx&=-\frac{1}{2}\int_{\TT}G_n'(v)v,_{\ell}(v,_{jki}v,_{jki}),_\ell\,dx\\
&=\frac{1}{2}\int_{\TT}(G_n'(v)v,_{\ell}),_\ell v,_{jki}v,_{jki}\,dx\\
&=\frac{1}{2}\int_{\TT}G_n''(v)v,_\ell v,_{\ell} v,_{jki}v,_{jki}\,dx
+\frac{1}{2}\int_{\TT}G_n'(v) v,_{\ell\ell} v,_{jki}v,_{jki}\,dx.
\end{align*}
Resuming
\begin{align*}
I_2&=-\int_{\TT}G_n''(v)v,_kv,_{\ell}v,_{ii\ell}v,_{jjk}\,dx
-\int_{\TT}G_n'(v)v,_{\ell k}v,_{ii\ell}v,_{jjk}\,dx+\int_{\TT}G_n''(v)v,_iv,_{\ell}v,_{jjk}v,_{i\ell k}\,dx\\
&\q
+\int_{\TT}G_n'(v)v,_{\ell i}v,_{jjk}v,_{i\ell k}\,dx
-\int_{\TT}G_n''(v)v,_jv,_{\ell}v,_{i\ell k}v,_{jki}\,dx
-\int_{\TT}G_n'(v)v,_{\ell j}v,_{i\ell k}v,_{jki}\,dx
\\
&\q+\frac{1}{2}\int_{\TT}G_n''(v)v,_\ell v,_{\ell} v,_{jki}v,_{jki}\,dx
+\frac{1}{2}\int_{\TT}G_n'(v) v,_{\ell\ell} v,_{jki}v,_{jki}\,dx.
\end{align*}
Since the last expression of $I_2$ is made up of integrals involving only first and third or second and third order derivatives, we just focus on the first two integrals in the r.h.s.. The inequalities in \eqref{l4} with $w=v$, \eqref{l42} with $w=\D v$ and the fact that $\|f\|_{W^{\al,\infty}}\le \|f\|_{A^\al}$ yield to
\begin{align*}
\int_{\TT}G_n''(v)v,_kv,_{\ell}v,_{ii\ell}v,_{jjk}\,dx&\le c\int_{\TT} |v,_k||v,_{\ell}||v,_{ii\ell}||v,_{jjk}|\,dx\\
&\le \norm{\N u}_{L^4}^2\norm{\N\D u}_{L^4}^2\\
&\le c\|v\|_{A^0} \| v\|_{H^2}^2\|v\|_{A^4},
\\
\int_{\TT}G_n'(v)v,_{\ell k}v,_{ii\ell}v,_{jjk}\,dx&\le c\int_{\TT} |v,_{\ell k}||v,_{ii\ell}||v,_{jjk}|\,dx \\
& \le c\|v\|_{H^2}\|\N\D v\|_{L^4}^2 \\
&\le c\| v\|_{H^2}^2\|v\|_{A^4}.
\end{align*}
Then, we estimate $I_2$ as
\[
I_2\le c(1+\|v\|_{A^0}) \| v\|_{H^2}^2\|v\|_{A^4}
\]
and, consequently, the main inequality in \eqref{abc} becomes
\[
\frac{d}{dt}\|v\|_{H^2}^2\le c(1+\|v\|_{A^0})\|v\|_{A^4}\pare{1+\| v\|_{H^2}^2}.
\]
Standard computations lead to
\begin{align*}
\|v(t)\|_{H^2}&\le e^{c(1+\|v\|_{L^\infty(A^0)})\int_0^t\|v(s)\|_{A^4}\,ds}\\
&\q\times
\left[
\|v_0\|_{H^2}+C(1+\|v\|_{L^\infty(A^0)})\int_0^t \|v(s)\|_{A^4}
e^{-(1+\|v\|_{L^\infty(A^0)})\int_0^s\|u(\tau)\|_{A^4}\,d\tau}\,ds
\right]\\
&\le e^{c(1+\|v\|_{L^\infty(A^0)})\int_0^t\|v(s)\|_{A^4}\,ds}
\left[
\|v_0\|_{H^2}+C(1+\|v\|_{L^\infty(A^0)})\int_0^t \|v(s)\|_{A^4}
\,ds
\right]\\
&\le c(\norm{v_0}_{A^0}),
\end{align*}
where the last inequality is due to \cref{teoex} (see \eqref{disA4} and comments below). In particular, we have deduced that $v$ is uniformly bounded in $L^\infty(0,T;H^2(\TT))$.\\

The regularity $L^2(0,T;H^4(\TT))$ follows integrating in time the inequality  in \eqref{star} and taking into account the estimates on the time-dependent terms $I_1,\,I_2,\,I_3$:
\begin{align}
\frac{1}{2}\|v(t)\|_{H^2}^2 +\iint_{(0,t)\times\TT}|\D^2u|^2\,dx\,ds&\le\int_0^tI_1(s)+I_1(s)+I_3(s)\,ds+\frac{1}{2}\|v_0\|_{H^2}^2\nonumber\\
&\le c(1+\|v\|_{L^\infty(A^0)})\pare{1+\| v\|_{L^\infty(H^2)}^2}\int_0^t\|v(s)\|_{A^4}\,ds+\frac{1}{2}\|v_0\|_{H^2}^2\nonumber\\
&\le c(\norm{v_0}_{A^0})\label{dis}.
\end{align}

\end{proof}

\begin{proposition}[Compactness results]\label{propcomp}
Let the parameters verify \eqref{poscost-bis}. Assume that $v_0\in A^0(\TT)\cap H^2(\TT)$ satisfies
\[
\norm{v_0}_{A^0}<1,
\]
and the smallness condition in \eqref{smallness} (where $\deu ,\,\ded $ have been defined in \eqref{de1}, \eqref{de2}). Then, up to subsequences, we  have that every approximating sequence of solutions $\{v_n\}_n$ of \eqref{appr} verifies
\begin{equation}\label{L2Hr}
 v_n\rightarrow v\q \t{ in}\q L^2(0,T;H^r(\TT)),\;0\leq r<4.
\end{equation}
Furthermore, the limit function $u$ satisfies
\begin{equation}\label{CH2}
v\in{C}([0,T];H^2(\TT)).
\end{equation}
\end{proposition}
\begin{proof}
H\"older's inequality implies
\begin{align*}
\int_0^t\|v_n(s)-v(s)\|_{H^r}^2\,ds&\le \int_0^t\|v_n(s)-v(s)\|_{H^4}^\frac{r}{2}\|v_n(s)-v(s)\|_{L^2}^{\frac{4-r}{2}}\,ds\\
&\le
\left(\int_0^t\|v_n(s)-v(s)\|_{H^4}^2\,ds\right)^{\frac{r}{4}}
\left( \int_0^t\|v_n(s)-v(s)\|_{L^2}^2\,ds\right)^{\frac{4-r}{4}}
\end{align*}
which converges to zero by \eqref{LL}, then \eqref{L2Hr} follows.

The continuity regularity \eqref{CH2} can be proved as in \cref{corregw}, namely using that
\[
\D v\in L^2(0,T;H^2(\TT))
\]
from  \cref{propreg}, and then proving that
\[
 \pat\D v\in L^2(0,T;H^{-2}(\TT)).
\]
\end{proof}

\begin{proof}[Proof of \cref{teoreg}]
The regularity \eqref{reg1} follows from \cref{propreg,propcomp}.\\
The energy inequality in \eqref{H2H4} is a consequence of \eqref{dis}.

\end{proof}

\section{Final remarks}

\subsection{On the uniqueness}
The techniques previously employed do not provide us with enough regularity to prove uniqueness results. Indeed, we  just have $v\in \mathcal{M}(0,T;W^{4,\infty}(\TT))$ but we would need
\[
v\in L^1(0,T;A^4(\TT))
\]
in order to  get a comparison result between sub and supersolution of \eqref{pb} reasoning as in \cref{prioriw}.\\
A possible approach is the one exploited by J.-G. Liu and R. Strain in \cite{LS}: here, under the assumption of medium size data belonging to $A^2$, the authors first proved the uniqueness of solutions to
\[
\pat u=\D e^{-\D u}\qq\text{in }(0,T)\times\R^N,
\]
and then they took advantage of this information to obtain uniqueness in $A^0(\R^N)$. We explicitly point out that the uniqueness class does not coincide with the existence class.

\subsection{A  fixed point technique}

Another possible way to approach problems of \eqref{pb} type is the one proposed by D. M. Ambrose in \cite{A}. Here, the author proves existence and analyticity results for a fourth order problem which describes crystal growth surfaces through fixed point techniques. The equation in object is \eqref{exp}.
The main contributions provided by \cite{A} concern both the proof of the analiticity and the improvement of the smallness condition on the $A^0$ semi-norm of the initial datum w.r.t. \cite{GM}, where the same problem is studied.\\
Hence, due to the nature of the model in \eqref{pb}, it could be interesting to try by means of similar arguments as in \cite{A}.

\section*{Acknowledgments}
The author is very grateful to the anonymous referees for their careful reading and for their significant contribution to improve the initial version of this paper.\\
The author thanks Rafael Granero-Belinchón for some stimulating discussions and comments.\\
This work  was  supported by  Grant RYC2021-033698-I, funded by the Ministry of Science and Innovation/State Research Agency/10.13039/501100011033 and by the European Union "NextGenerationEU/Recovery, Transformation and Resilience Plan".

\end{document}